\documentclass[12pt]{amsart}
\input {epsf}
\usepackage{latexsym, amsbsy, amsmath, amsfonts, amssymb, amsthm, amscd, amscd}
\usepackage[all]{xy}
\newcommand{\R}{\mathbb R}

\newcommand{\Z}{\mathbb Z}

\newtheorem{Theorem}{Theorem}[section]

\newtheorem{Proposition}{Proposition}[section]

\allowdisplaybreaks[2]
\numberwithin{equation}{section}
\numberwithin{figure}{section}
\title{On the inverse braid and reflection monoids of type $B$}
\author[Vershinin]{V.~V.~Vershinin}
\address{D\'epartement des Sciences Math\'ematiques,
                                     Universit\'e Montpellier II,
Place Eug\`ene Bataillon,
34095 Montpellier cedex 5, France}
\email{ vershini@math.univ-montp2.fr}
\address{Sobolev Institute of Mathematics, Novosibirsk 630090,
Russia } 
\email{ versh@math.nsc.ru}
\begin{document}
\begin{abstract}
There are well known relations between braid groups and symmetric groups,
between Artin-Briskorn braid groups and Coxeter groups.
Inverse braid monoid the same way is related to the inverse symmetric
monoid.
In the paper we show that similar relations exist between
the inverse braid monoid of type $B$ and the inverse reflection
monoid of type $B$. This gives a presentation of the last monoid.
\end{abstract}
\maketitle
\tableofcontents

\section{Introduction}

Let $V$ be a finite dimensional real vector space
($\operatorname{dim} V= n$) with Euclidean structure. Let $W$ be a
finite subgroup of $GL(V)$ generated by reflections. 
We suppose that $W$ is {\it essential}, i.e. that the set of fixed vectors
by the action of $W$ consists only of zero: $V^W=0$. 
Let
$\mathcal M$ be the set of hyperplanes such that $W$ is generated by
orthogonal reflections with respect to $M \in \mathcal M$. We suppose
that for every $w\in W$ and every hyperplane $M \in \mathcal M$ the
hyperplane $w(M)$ belongs to $\mathcal M$. 

Consider the complexification $V_C$ of the space $V$
and the complexification $M_C$ of $M\in\mathcal M$. Let
$Y_W=V_C-\bigcup_{M\in\mathcal M}M_C$. The group $W$ acts freely 
on $Y_W$. Let $X_W=Y_W/W$
then $Y_W$ is a covering over $X_W$ corresponding to the group
$W$.

This generalized braid  group  $Br(W)$ corresponding to the Coxeter group $W$ 
is defined as the
fundamental group of the space $X_W$ of regular orbits of the action of
$W$ and the corresponding pure braid group $P(W)$ is defined as the
fundamental group of the space $Y_W$.  So, for the generalized braid groups
is $Br(W)=\pi_1(X_W)$, $P(W)=\pi_1(Y_W)$.
The groups  $Br(W)$  were defined by 
E.~Brieskorn \cite{Bri1}, and are called also as  Artin--Brieskorn groups.
E.~Brieskorn \cite{Bri1} and P.~Deligne  \cite{Del} proved that the spaces $X_W$ and $Y_W$  are of the type $K(\pi, 1)$. 
 
The covering which corresponds to the action of $W$ on $Y_W$ gives rise to the
exact sequence 

\begin{equation*}
1 \rightarrow\pi_1( Y_W)\buildrel p_*
\over\rightarrow \pi_1(X_W) \rightarrow W
\rightarrow 1 .
\end{equation*}

So, there is a naturally defined map $\rho: Br(W)\to W$.

Geometrical braid, as a system of $n$ curves in $\R^3$ lead to the notion
of partial braid where several among these $n$ curves can be omitted; partial braids form  the {\it inverse braid monoid }  $IB_n$  \cite{EL}. 
By definition a monoid is {\it inverse} if for any element $a$ of it 
there exists a unique element $b$ (which is called {\it inverse}) such that 
\begin{equation*}
a = aba
\end{equation*}
 and 
\begin{equation*}
b = bab.
\end{equation*}
This notion  was introduced by V.~V.~Wagner in 1952 \cite{Wag}. 
 See the books \cite{Pet} and \cite{Law} as general references for inverse semigroups.

The multiplication of partial braids is shown at Figure~\ref{fi:vv2}.
At the last stage it is necessary to remove any arc that does not join the upper 
or lower planes.

\begin{figure}
\epsfbox{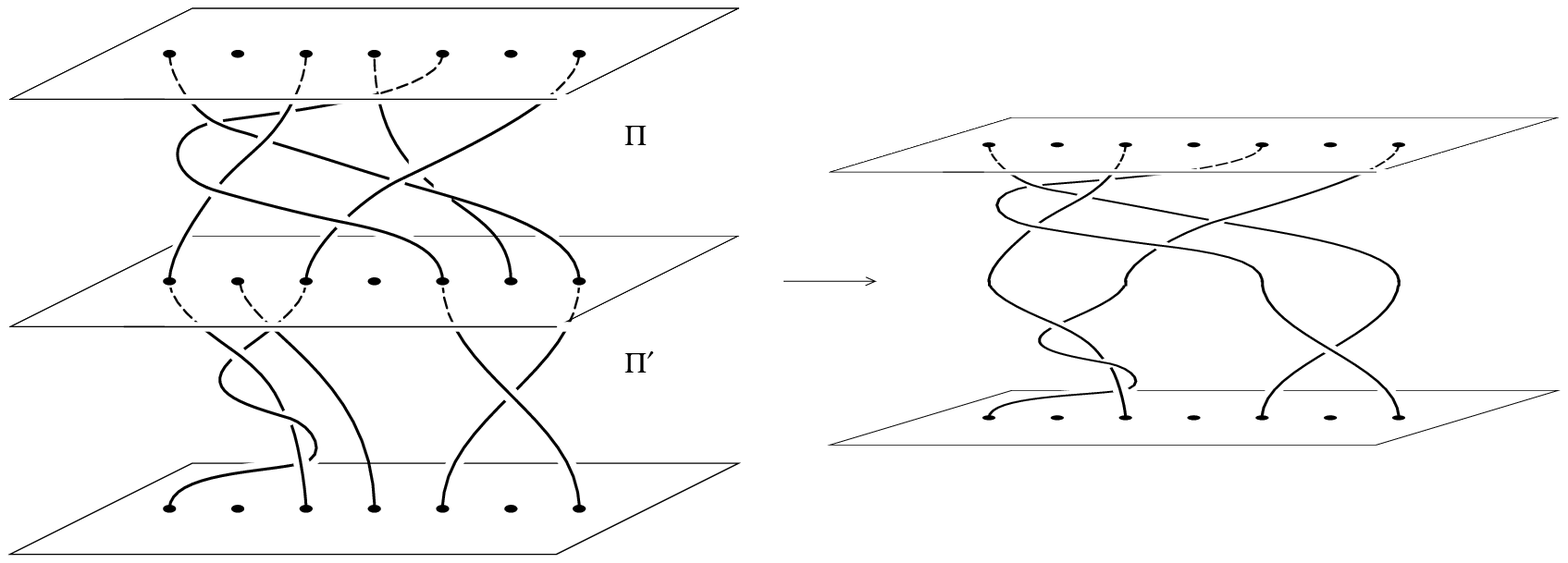}
\caption{} \label{fi:vv2}
\end{figure}

So, the classical braid group (which corresponds to  $W=\Sigma_n$, symmetric 
group) is included into the inverse braid monoid   $IB_n$.
 
The most important example of an inverse monoid is a monoid of partial (defined on a subset)
injections of a set into itself. For a finite set this gives us the notion of a
{\it symmetric inverse monoid } $I_n$ which generalizes and includes the classical
symmetric group $\Sigma_n$. A presentation of symmetric inverse monoid was
obtained by L.~M.~Popova \cite{Po}, see also formulas 
(\ref{eq:brelations}-\ref{eq:syminvrelations})
below.

Now let $W$ be the Coxeter group of type $B_n$. The corresponding 
inverse braid monoid $IB(B_n)$ was studied in \cite{Ve11} and the 
reflection monoid $I(B_n)$ in \cite{EF}.

The aim of the present paper is to show that in the case of type $B$ the 
situation is quite similar: there exists a map
$\rho_B: IB(B_n)\to I(B_n)$ such that the following diagram 
\begin{equation}
\CD
Br(B_n)@>>>W(B_n)\\
\downarrow&&\downarrow\\
IB(B_n)@>\rho_B>> I(B_n) 
\endCD
\label{eq:cod}
\end{equation}
(where the vertical arrows mean inclusion of the group of invertible elements
into a monoid) is commutative.

\section{Inverse braid monoid and type $B$\label{sec:typeb}}

Let $N$ be a finite set of cardinality $n$, say $N= \{v_1, \dots, v_n\}$.
The inverse symmetric monoid $I_n$ can be interpreted as a
monoid of partial monomorphisms of $N$ into itself. Let us equip elements of
$N$ with the signs, i.e. let $SN = \{\delta_1 v_1, \dots, \delta_n v_n\}$,
where $\delta_i =\pm 1$. The Weyl group $W(B_n)$ of type $B$ can be 
interpreted as 
a group of signed permutations of the set $SN$: 
\begin{equation*}
W(B_n)=\{\sigma - \text{
bijection \ of } SN: (-x)\sigma  =-(x)\sigma \text{ for} \ x\in SN\}. 
\end{equation*}
The {\it monoid of partial signed permutations} $I(B_n)$ 
is defined as follows
\begin{multline*}
I(B_n)=\{\sigma - \text{partial \
bijection \ of } SN: (-x)\sigma  =-(x)\sigma \text{ for} \ x\in SN \\
\text {and } x\in \operatorname{dom} \sigma \ \text {if \, and \, only \, if} \
-x\in \operatorname{dom} \sigma\}, 
\end{multline*}
where $\operatorname{dom} \sigma$ means a domain of definition of
the monomorphism $\sigma$.

We remind that a monoid $M$ is {\it factorisable} if $M= EG$ where 
$E $ is a set of idempotents of $M$ and $G$ is a subgroup of $M$. 
Evidently the monoid $I(B_n)$ is factorisable \cite{EF}
as every partial signed permutation can be extended to an element 
of the group of units of $I(B_n)$ i.e. a signed permutation with the 
domain equal to $SN$.

Usually  the braid group $Br_n$ 
is given by the following Artin presentation \cite{Art1}.
It has the generators $\sigma_i$, 
$i=1, ..., n-1$, and two types of relations: 
\begin{equation}
 \begin{cases} \sigma_i \sigma_j &=\sigma_j \, \sigma_i, \ \
\text{if} \ \ |i-j|
>1,
\\ \sigma_i \sigma_{i+1} \sigma_i &= \sigma_{i+1} \sigma_i \sigma_{i+1} \, .
\end{cases} \label{eq:brelations}
\end{equation}

The following presentation for the inverse braid monoid was obtained in
\cite{EL}. It has the generators $\sigma_i, \sigma_i^{-1} $, $i=1,\dots,n-1,$
$\epsilon$, and relations
\begin{equation}
 \begin{cases} 
&\sigma_i\sigma_i^{-1}=\sigma_i^{-1}\sigma_i =1, \ \text {for \ all} \ i, \\
&\epsilon \sigma_i  =\, \sigma_i \epsilon \ \ \text {for } i\geq 2,   \\ 
&\epsilon\sigma_1 \epsilon  = \sigma_{1} \epsilon \sigma_1 \epsilon = 
\epsilon\sigma_{1} \epsilon \sigma_1, \\
&\epsilon = \epsilon^2 = \epsilon \sigma_1^2= \sigma_1^2 \epsilon
\end{cases} \label{eq:invbrelations}
\end{equation}
and the braid relations (\ref{eq:brelations}).

Geometrically the generator $\epsilon$ means that the first string in the trivial braid is absent.
 
If we replace the first relation in (\ref{eq:invbrelations})
 by the following set of relations
\begin{equation}
\sigma_i^2 =1, \ \text {for \ all} \ i, \\
 \label{eq:syminvrelations}
\end{equation}
and delete the
superfluous relations 
$$\epsilon =  \epsilon \sigma_1^2= \sigma_1^2 \epsilon, $$
we get a presentation of the symmetric inverse monoid $I_n$ \cite{Po} . 
We also can simply add the relations (\ref{eq:syminvrelations})
if we do not worry about redundant relations.
We get a canonical map \cite{EL}
\begin{equation*}
\rho_n: IB_n\to I_n
\end{equation*}
which is a natural extension of the corresponding map for the braid 
and symmetric groups.

More balanced relations for the inverse braid monoid were 
obtained in \cite{Gil}.
Let $\epsilon_i$ denote the trivial braid with $i$th string deleted,
formally:
\begin{equation*}
\begin{cases} \epsilon_1 &= \epsilon, \\
\epsilon_{i+1} &= \sigma_i^{\pm 1}\epsilon_{i}\sigma_i^{\pm 1}.  
\end{cases} \end{equation*} 
So, the generators are: $\sigma_i, \sigma_i^{-1} $, $i=1,\dots,n-1,$
$\epsilon_i$, $i=1,\dots,n$, and relations  are the following:

\begin{equation}
 \begin{cases} 
&\sigma_i\sigma_i^{-1}=\sigma_i^{-1}\sigma_i =1, \ \text {for all} \ i, \\
&\epsilon_j \sigma_i =\, \sigma_i \epsilon_j \ \ \text {for } \ j \not= i, i+1, \\ 
&\epsilon_i\sigma_i =  \sigma_{i} \epsilon_{i+1},  \\
&\epsilon_{i+1}\sigma_i =  \sigma_{i} \epsilon_{i},  \\
&\epsilon_i = \epsilon_i^2 , \\
& \epsilon_{i+1} \sigma_i^2= \sigma_i^2 \epsilon_{i+1} = \epsilon_{i+1}, \\
&\epsilon_i \epsilon_{i+1} \sigma_i = \sigma_{i} \epsilon_i \epsilon_{i+1}
=\epsilon_i\epsilon_{i+1},
\end{cases} \label{eq:invbrelations2}
\end{equation}
plus the braid relations (\ref{eq:brelations}).

Let $E F_n$ be a monoid of partial isomorphisms of a free group
$F_n$ defined as follows. Let $a$ be an element of the symmetric
inverse monoid $I_n$, $a\in I_n$, $J_k =\{j_1, \dots, j_k\}$ 
is the image of $a$, and elements $i_1, \dots, i_k$ belong to
domain of the definition of $a$. The monoid $E F_n$
consists of isomorphisms
$$<x_{i_1}, \dots, x_{i_k}> \, \to \ <x_{j_1}, \dots, x_{j_k}>$$
expressed by 
$$f_a :x_i\mapsto w_i^{-1} x_{a(i)}w_i, $$
if $i$ is among $i_1, \dots, i_k$ and not defined otherwise and 
$w_i$ is a word on $x_{j_1}, \dots, x_{j_k}$.
The composition of $f_a$ and $g_b$, $a, b\in I_n$, 
is defined for $x_i$ belonging to the domain of $a\circ b$.
We put $x_{j_m}=1$ in a word $w_i$ if $x_{j_m}$ does not belong
to the domain of definition of $g$.
If we put $w_i=1$ we get an inclusion of $ I_n$ into
$EF_n$. Sending each $f_a\in EF_n$ to $a\in  I_n$
we get a homomorphism $ EF_n \to  I_n$.
\begin{Proposition}
The canonical maps $I_n \to EF_n$ and $EF_n \to  I_n$
give the following splitting
\begin{equation*}
I_n \to EF_n \to  I_n.
\end{equation*}
\end{Proposition}
\hfill $\square$

We remind that the Artin-Brieskorn braid group of the type $B$ is 
isomorphic to the braid group of a punctured disc \cite{La}, \cite{Ve1}, \cite{Ve6}.
  With respect to the classical braid group it has an extra generator
$\tau$ and the relations of type $B$:
\begin{equation}
\begin{cases} \tau\sigma_1\tau\sigma_1 &= \sigma_1\tau\sigma_1\tau, \\ 
\tau\sigma_i &= \sigma_i\tau, \ \  \text{if} \ \ i>1,\\ 
\sigma_i \sigma_{i+1} \sigma_i &= \sigma_{i+1} \sigma_i \sigma_{i+1}, \\
\sigma_i \sigma_j &=\sigma_j \, \sigma_i, \ 
\text{if} \ \ |i-j|>1.
\end{cases} \label{eq:typbrel}
\end{equation}
The monoid $IBB_n$ of partial braids of the type $B$
 can be considered also as a submonoid of $IB_{n+1}$ consisting of
partial braids with the first string fixed. An interpretation  as
a monoid of isotopy classes of maps is possible as well. As usual consider a disc $D^2$ with given $n+1$  points. Denote the set of these points by $Q_{n+1}$.
Consider homeomorphisms
 of $D^2$ onto a copy of the same disc with the condition that
 the first point is always mapped into itself and among the other $n$
 points
only $k$ points,  $k \leq n$ (say $i_1, \dots, i_k$) are mapped
bijectively onto the $k$ points (say $j_1, \dots, j_k$) of the 
set $Q_{n+1}$ (without the first point) of second copy of $D^2$. 
The isotopy classes of such homeomorphisms form the monoid $IBB_n$.

\begin{Theorem}\cite{Ve11}  We get a presentation of the monoid 
${IB}(B_n)$ if we add to
the presentation of the braid group of type $B$ (\ref{eq:typbrel}) 
the generator
$\epsilon$ and the following relations
\begin{equation}
 \begin{cases} 
&\tau\tau^{-1}=\tau^{-1}\tau =1,  \\
&\sigma_i\sigma_i^{-1}=\sigma_i^{-1}\sigma_i =1, \ \text {for all} \ i, \\
&\epsilon \sigma_i  =\, \sigma_i \epsilon \ \ \text {for } i\geq 2,   \\ 
&\epsilon\sigma_1 \epsilon  = \sigma_{1} \epsilon \sigma_1 \epsilon = 
\epsilon\sigma_{1} \epsilon \sigma_1, \\
&\epsilon = \epsilon^2 = \epsilon \sigma_1^2= \sigma_1^2 \epsilon\\
&\epsilon\tau = \tau\epsilon = \epsilon.
\end{cases} \label{eq:IBB}
\end{equation} 
We get another presentation of the monoid 
${IB}(B_n)$ if we add to 
the presentation (\ref{eq:invbrelations2})
of  $IB_n$ one generator $\tau$, the type $B$ relations
(\ref{eq:typbrel}) and the following relations
\begin{equation*}
 \begin{cases} 
&\tau\tau^{-1}=\tau^{-1}\tau =1,  \\
&\epsilon_1\tau = \tau\epsilon_1 = \epsilon_1.
\end{cases} 
\end{equation*}
It is a factorisable inverse monoid.
\end{Theorem}

 We define an action of $ {IB}(B_n)$ on $SN$ by partial 
 isomorphisms  as follows
 \begin{equation}
 \sigma_i(\delta_j v_j) =\begin{cases} \delta_{i} v_{i+1}, \ \text{if } j=i,\\
  \delta_{i+1} v_i, \ \text{if } j=i+1,\\
  \delta_j v_j,  \ \text{if } j\not=i,i+1, \\
  \end{cases}
  \label{eq:sigmai}
 \end{equation}
 \begin{equation}
 \tau(\delta_j v_j) =\begin{cases} -\delta_{1} v_{1}, \ \text{if } j=1,\\
  \delta_j v_j,  \ \text{if } j\not=1,\\
  \end{cases}
 \end{equation} 
 \begin{equation}
 \operatorname{dom}\epsilon =\{\delta_2 v_2, \dots, \delta_n v_n\},
 \end{equation}  
 \begin{equation}
 \epsilon(\delta_j v_j) =  \delta_j v_j,  \ \text{if } j=2, \dots, n,\\
 \end{equation} 
  \begin{equation}
 \operatorname{dom}\epsilon_i =\{\delta_1 v_1, \dots, {\widehat{\delta_i v_i}}, 
 \dots, \delta_n v_n\},
 \end{equation}  
 \begin{equation}
 \epsilon_i(\delta_j v_j) =  \delta_j v_j,  \ \text{if } j=1, \dots, \widehat{i},
 \dots, n.\\
  \label{eq:epsiloni}
 \end{equation} 
 Direct checking shows that the relations of the inverse braid monoid of type 
 $B$ are satisfied by the compositions of partial isomorphisms defined by
  $\sigma_i$, $\tau$  and $\epsilon_i$.
\begin{Theorem}  The action given by the formulas (\ref{eq:sigmai} -  
\ref{eq:epsiloni}) defines a homomorphism of inverse monoids
$\rho_B: IB(B_n)\to I(B_n)$ such that the  diagram (\ref{eq:cod})
is commutative.
\end{Theorem} \hfill $\square$
\begin{Theorem} The homomorphism 
$\rho_B: IB(B_n)\to I(B_n)$
is an epimorphism.
We get a presentation of the monoid $ I(B_n)$ if in
a presentation of ${IB}(B_n)$ 
 we replace the first relation in (\ref{eq:invbrelations})
 by the following set of relations
\begin{equation*}
\sigma_i^2 =1, \ \text {for \ all} \ i, \\
\end{equation*}
and delete the superfluous relations 
$$\epsilon =  \epsilon \sigma_1^2= \sigma_1^2 \epsilon, $$
and 
we replace the first relation in (\ref{eq:IBB})
 by the following relation
\begin{equation*}
\tau^2 =1.  \\
\end{equation*}
\end{Theorem}
\begin{proof} Let us temporarily denote by ${IB}_n$  the 
monoid with the presentation given in the statement of Theorem. 
To see that the homomorphism $\rho_B$ is an epimorphism we use 
the fact that
the monoid ${I}(B_n)$ is factorisable, so its every
element can be written in the form $\epsilon g$ where $\epsilon$
belong to the set of idempotents and $g$ is an element of the 
the Weyl group of type $B$, $W(B_n)$. For the Weyl group
the map $\rho_B$ is an epimorphism $W(B_k)= Br(B_k)/P(B_k)$
and the sets of idempotents for the monoids ${IB}(B_n)$
and ${I}(B_n)$ coincide and the map $\rho_B$ restricted
to $E({IB}(B_n))$ is identity.

It follows from the definition of the action that $\tau^2$ and 
$\sigma_i^2$ are mapped to the unit by the map $\rho_B$.
So the homomorphism $\rho_B$ is factorised by the homomorphism
$\tilde{\rho}_B : IB_n \to I(B_n)$: 
\begin{equation*}
\rho_B: IB(B_n)\to {IB}_n \to I(B_n).
\end{equation*}
To show that $\tilde{\rho}_B$ is an isomorphism we compare the
cardinalities of $IB_n$ and $I(B_n)$. 
It is easy to calculate that the cardinality of 
$ I(B_n)$ is equal to 
${ \sum_{k=0}^{n}2^k\left(\begin{array}{c}n\\k\end{array}\right)^2
k!}$.

Let $\epsilon_{k+1, n}$ denote the partial braid with the trivial first $k$ strings 
and absent of the rest $n-k$ strings. It can be expressed using the generator $\epsilon$ 
or the generators $\epsilon_i$ as follows
\begin{equation*} 
\epsilon_{k+1, n}=
\epsilon\sigma_{n-1}\dots\sigma_{k+1}\epsilon \sigma_{n-1}\dots\sigma_{k+2}
\epsilon\dots \epsilon\sigma_{n-1}\sigma_{n-2}\epsilon \sigma_{n-1}\epsilon,
\end{equation*}
\begin{equation*} \epsilon_{k+1, n}=
\epsilon_{k+1}\epsilon_{k+2} \dots \epsilon_{n}.
\end{equation*}
It was proved in \cite{EL} that
every partial braid has a representative of the form
\begin{equation*} 
\sigma_{i_1}\dots\sigma_{1}\dots \sigma_{i_k}\dots\sigma_{k}
\epsilon_{k+1, n}x \epsilon_{k+1, n}\sigma_{k}\dots\sigma_{j_k}\dots\sigma_{1}
\dots\sigma_{j_1},
 \\ 
\end{equation*}
\begin{equation*} k\in \{0,\dots, n\}, 
0\leq i_1<\dots<i_k\leq n-1  \ 
\text{and} \  0\leq j_1<\dots<j_k\leq n-1,
\end{equation*}
where $x\in Br_k$. 
The same is true  for ${IB}(B_n)$, where $x\in Br(B_n)$.
The elements $\tau^2$ and $\sigma_i^2$ are mapped to $1$ by $\rho_B$,
so each equivalence class  modulo pure braid group of the type $B_k$
is mapped to the same element in ${I}(B_n)$. These equivalence
classes form 
the Weyl group $W(B_k)$.
The order of the Weyl group of type $B_k$ is equal to
${ 2^k k!}$. 
We see that the set
of cardinality less or equal than 
${ \sum_{k=0}^{n}2^k\left(\begin{array}{c}n\\k\end{array}\right)^2
k!}$ is mapped epimorphically onto the set of exactly this cardinality.
It means that the epimorphism $\tilde{\rho}_B$ is an isomorphism.
\end{proof}

Let $\mathcal E$ be the monoid generated by one idempotent generator 
$\epsilon$. 
\begin{Proposition} The abelianization $Ab(IBB_n)$ of the monoid
 $IBB_n$ is isomorphic to
$\mathcal E \oplus \Z^2$, factorised by the relations
\begin{equation*}
\begin{cases}
\epsilon + \tau =\epsilon,\\
\epsilon + \sigma =\epsilon, \\
\end{cases}
\end{equation*}
where $\tau$ and $\sigma$ are generators of $\Z^2$. 
 The canonical map
\begin{equation*}
a: IBB_n \to Ab(IBB_n)
\end{equation*}
is given by the formulas:
\begin{equation*}
\begin{cases}
a(\epsilon_i) = \epsilon ,\\
a(\tau) = \tau ,\\
a(\sigma_i) = \sigma .
\end{cases}
\end{equation*}
The canonical map from $Ab(IBB_n)$ to $Ab(I(B_n))$ consists of 
factorising $\Z^2$ modulo $2$.
\end{Proposition}
\hfill $\square$

\end{document}